\documentclass[11pt]{amsart}
\usepackage[margin=1.15in]{geometry}
\usepackage{amsmath,amssymb,amsthm}
\usepackage{booktabs}
\usepackage{xcolor}
\usepackage[colorlinks=true,linkcolor=blue!60!black,citecolor=blue!60!black,urlcolor=blue!60!black]{hyperref}

\newtheorem{theorem}{Theorem}[section]
\newtheorem{lemma}[theorem]{Lemma}
\newtheorem{proposition}[theorem]{Proposition}

\theoremstyle{definition}
\newtheorem{definition}[theorem]{Definition}
\theoremstyle{remark}
\newtheorem{remark}[theorem]{Remark}

\newcommand{\Rea}{\operatorname{Re}}
\newcommand{\Ima}{\operatorname{Im}}
\newcommand{\qbin}[2]{\genfrac{[}{]}{0pt}{}{#1}{#2}}

\title[An explicit cubic wedge for Toeplitz minors of the $\xi$-coefficients]{An explicit uniform cubic wedge for consecutive Toeplitz minors of the Riemann $\xi$-coefficients}

\author{Wojciech Micha{\l}owski}
\email{michalowski.wojciech1@gmail.com}

\date{18 July 2026}

\begin{document}

\begin{abstract}
Let $\Xi(x)=\tfrac18\,\xi\!\left(\tfrac12+\tfrac{ix}{2}\right)
=\sum_{k\ge0}(-1)^k a_k x^{2k}$ be the normalized Riemann $\xi$-function,
so that $a_k=\frac1{(2k)!}\int_0^\infty u^{2k}\Phi(u)\,du>0$ with the
classical kernel $\Phi$, and let
$D_{r,k}=\det\,[a_{k+j-i}]_{i,j=0}^{r-1}$ denote the consecutive Toeplitz
minors of the coefficient sequence.  The Riemann Hypothesis is equivalent
to $(a_k)$ being a P\'olya frequency sequence of infinite order, and hence
to nonnegativity of all Toeplitz minors.  We prove the explicit,
uniform-in-$r$ theorem
\[
D_{r,k}>0
\qquad\text{for every } r\ge2 \text{ and every } k\ge10^{18}\,r^{3}.
\]
The result gives an explicit cubic tail scale and is uniform in the order
$r$, in contrast with the asymptotic fixed-order positivity of Katkova.
The proof does not use any verified numerical zeros of $\zeta$.  It combines
(A)~a certified complex saddle-point analysis of the moment transform
$\int_0^\infty u^{2z}\Phi(u)\,du$, giving a zero-free relative disk and the
uniform all-degree coefficient bound
$\bigl|f^{(d)}(k)/d!\bigr|\le3\cdot40^{d}k^{1-d}$ ($d\ge3$) for
$f=\log\bigl(\int u^{2z}\Phi/\Gamma(2z+1)\bigr)$;
(B)~an exact $q$-Pascal dilation semigroup which bounds the
signature-whitened response of every degree simultaneously; and
(C)~a weighted Banach-algebra majorant for the nonlinear remainder,
closed by an inertia-preservation argument for an explicit indefinite
comparison model.  All analytic constants are certified in Arb ball
arithmetic with directed rounding, and all algebraic identities are
verified in exact rational arithmetic; ancillary files reproduce every
certificate.  The result is a statement about the tail regime
$k\gg r^{3}$ only; it makes no progress on the Riemann Hypothesis, which
concerns the complementary region.
\end{abstract}

\maketitle

\section{Introduction}\label{sec:intro}

\subsection{The coefficient reformulation}
Write $\xi(s)=\tfrac12 s(s-1)\pi^{-s/2}\Gamma(s/2)\zeta(s)$ and let
\begin{equation}\label{eq:Phi}
\Phi(u)=\sum_{n\ge1}\bigl(2\pi^2n^4e^{9u}-3\pi n^2e^{5u}\bigr)
e^{-\pi n^2e^{4u}},\qquad u\ge0,
\end{equation}
be the classical Jacobi-theta kernel, so that
$\tfrac18\,\xi\!\left(\tfrac12+\tfrac{ix}2\right)
=\int_0^\infty\Phi(u)\cos(xu)\,du$.
Each summand of \eqref{eq:Phi} equals
$\pi n^2e^{5u}(2\pi n^2e^{4u}-3)e^{-\pi n^2e^{4u}}>0$, so $\Phi>0$.
The Maclaurin coefficients
\begin{equation}\label{eq:ak}
a_k=\frac1{(2k)!}\,m_{2k},\qquad
m_{2k}=\int_0^\infty u^{2k}\Phi(u)\,du,
\end{equation}
are positive, and $G(z):=\sum_{k\ge0}a_kz^k=\tfrac18\,
\xi\!\left(\tfrac12+\tfrac{\sqrt z}2\right)$ is entire of order
$\tfrac12$, hence of genus zero.

By the Aissen--Schoenberg--Whitney--Edrei characterization of
P\'olya frequency sequences \cite{ASW,Edrei} combined with the genus-zero
factorization, the Riemann Hypothesis is \emph{equivalent} to the
statement that $(a_k)_{k\ge0}$ is a P\'olya frequency sequence of
infinite order, i.e.\ that every minor of the one-sided Toeplitz matrix
$[a_{j-i}]$ is nonnegative.  We study the \emph{consecutive} minors
\begin{equation}\label{eq:Drk}
D_{r,k}:=\det\,[a_{k+j-i}]_{i,j=0}^{r-1},
\qquad r\ge1,\ k\ge0 .
\end{equation}
Here and throughout, $a_\ell=0$ for $\ell<0$, the standard one-sided
Toeplitz convention.
Strict positivity of all $D_{r,k}$ implies $\mathrm{PF}_\infty$ by the
one-sided strict consecutive-minor criterion (Katkova
\cite[Theorem~D]{Katkova}); a single certified negative $D_{r,k}$ would
disprove RH.  The consecutive minors are therefore the natural
finite-dimensional observables of the problem.  (In the heat-flow
deformation of this picture the de Bruijn--Newman constant is
non-negative \cite{RT}, so the $t=0$ sequence studied here is the
boundary case of the deformed family.)

\subsection{Known positivity regions}
Three families of results are known.  Csordas, Norfolk and Varga
\cite{CNV} proved the Tur\'an inequalities, which in the normalization
\eqref{eq:ak} read $a_k^2>\frac{k+1}{k}a_{k-1}a_{k+1}$; in particular
$D_{2,k}>0$ for all $k\ge1$.  Katkova \cite{Katkova} proved that for
every fixed order $m$ there is an $N(m)$ with $D_{m,k}>0$ for all
$k\ge N(m)$; her argument is asymptotic and yields no explicit or uniform
$N(m)$.  Finally, Schoenberg's sector theorem \cite{Schoenberg} (an
entire function of genus $0$ with positive coefficients whose zeros avoid
the sector $|\arg w|<\pi m/(m+1)$ generates a $\mathrm{PF}_m$ sequence)
combined with the verified height
$H=3{,}000{,}175{,}332{,}800$ of Platt and Trudgian \cite{PT} yields
$D_{r,k}\ge0$ for all $k$ whenever $r\le\lfloor\pi H\rfloor-1
\approx9.4\cdot10^{12}$.  Indeed, a zero
$\rho=\beta+i\gamma$ of $\xi$ corresponds to
$w=4(\rho-\tfrac12)^2$, a zero of $G$.  For $|\gamma|\le H$ the
verification gives $\beta=\tfrac12$, so $w<0$.  For $|\gamma|>H$,
using $|\beta-\tfrac12|<\tfrac12$,
\[
 \pi-|\arg w|
 =2\arctan\frac{|\beta-\tfrac12|}{|\gamma|}
 <\frac1H.
\]
Thus every zero avoids $|\arg w|<\pi m/(m+1)$ whenever
$m+1\le\pi H$.  This route consumes the numerical verification of zeros
and cannot reach orders beyond the verified height.

In the related Jensen-polynomial hierarchy, hyperbolicity of almost all
shifted Jensen polynomials is known \cite{GORZ}; we do not use
Jensen-polynomial methods, and Theorem~\ref{thm:main} neither follows
from nor implies such hyperbolicity results.

The present paper proves an explicit wedge which is uniform in the order
and independent of all zero verification.

\begin{theorem}\label{thm:main}
For every integer $r\ge2$ and every integer $k\ge10^{18}\,r^{3}$,
\[
D_{r,k}>0 .
\]
\end{theorem}

The constant $10^{18}$ is deliberately generous; no attempt was made to
optimize it.  The content of the theorem is the explicit cubic tail scale
together with uniformity in $r$.  In the proof below the third-order
centered logarithmic remainder and the signature-whitened response combine
through the dimensionless quantity $r^3/k$; we make no claim that the
constant or the exponent is optimal.

\begin{remark}[Scope]\label{rem:scope}
Theorem~\ref{thm:main} is a statement about the tail regime
$k\ge10^{18}r^3$ and makes \emph{no} progress on the Riemann Hypothesis.
The RH-critical regime is $k\sim r$ (indeed $\mathrm{PF}_\infty$ requires
all minors, and the difficulty concentrates at bounded $k/r$); the region
$k<10^{18}r^3$ remains completely open.  We state this prominently
because experience shows that partial positivity results in this area are
easily over-read.
\end{remark}

\subsection{Method and certification}
The proof has three independent gates and one assembly step.

\emph{Gate A} (\S\ref{sec:gateA}).  Let
$I(z)=\int_0^\infty u^{2z}\Phi(u)\,du$, $a(z)=I(z)/\Gamma(2z+1)$ (so
$a(k)=a_k$) and $f=\log a$.  We construct a certified steepest-descent
analysis of $I(z)$, uniform on the relative disks $|z-k|\le0.05k$ for
every integer $k\ge10^9$: the integrand's saddle continues analytically,
the deformed contour is dominated by an explicit complex-Gaussian core,
and $I(z)$ is zero-free on the disk.  A second-derivative Cauchy argument
(which avoids the spurious $\log k$ of a naive Cauchy estimate on $f$)
then yields the uniform all-degree bound
\begin{equation}\label{eq:cd-intro}
\Bigl|\frac{f^{(d)}(k)}{d!}\Bigr|\le3\cdot40^{d}\,k^{1-d},
\qquad d\ge3,\ k\ge10^9 .
\end{equation}

\emph{Gate B} (\S\ref{sec:gateB}).  The local comparison model for the
normalized coefficients is the log-quadratic sequence
$c_s=q_k^{s(s-1)/2}$ with $q_k=a_{k-1}a_{k+1}/a_k^2=e^{-\tau_k}$.  Its
reversed Hankel block is an explicit diagonal congruence of the
$q$-Vandermonde matrix $V_{ij}=q_k^{ij}$, whose exact $LDL^{T}$
factorization has $q$-binomial $L$ and pivots of alternating sign.  The
one-step dilation $R_1=L^{-1}\operatorname{diag}(q^i)L$ is exactly lower
bidiagonal, and extends to a one-parameter group whose whitened generator
has norm at most $4\sqrt{r\tau_k}$.  This bounds the signature-whitened
response of \emph{every} polynomial degree simultaneously.

\emph{Gate C} (\S\ref{sec:gateC}).  The true coefficients differ from the
model by $e^{h_s}$ with an explicit centered remainder $h$.  A weighted
$\ell^1$ Banach algebra turns \eqref{eq:cd-intro} and the Gate~B response
weights into the certified bound
$\|e^{h}-1\|\le e^{0.1310721}-1<0.14005<1$ for $k\ge10^{18}r^3$.

\emph{Assembly} (\S\ref{sec:assembly}).  An inertia-preservation lemma
shows that a whitened perturbation of norm $<1$ cannot change the
signature of the indefinite comparison model; an exact sign computation
(reversal sign $(-1)^{r(r-1)/2}$ against model inertia
$(\lceil r/2\rceil,\lfloor r/2\rfloor)$) then gives $D_{r,k}>0$.

The paper is computer-assisted in the following precise sense.  Every
inequality between explicit real constants is certified in Arb ball
arithmetic \cite{Johansson} (via \texttt{python-flint}) with directed
rounding, and every algebraic identity (the $q$-Pascal factorizations,
the bidiagonal dilation, the semigroup law) is verified in exact rational
arithmetic.  The ancillary files contain the complete certificate code
and a regression suite; Table~\ref{tab:certificates} maps each lemma to
its certificate.  No numerical zeros of $\zeta$, no floating-point
heuristics, and no unverified global optimization enter the proof.

The author's companion note \cite{PF5} contains a certified
\emph{negative} result of the same flavour (failure of order-five total
positivity for the continuous de Bruijn--Newman kernel); the two results
are logically independent, since coefficient positivity
\eqref{eq:Drk} and kernel total positivity are distinct properties.

\section{The curvature window}\label{sec:tau}

Throughout, $q_k=a_{k-1}a_{k+1}/a_k^2$ and $\tau_k=-\log q_k$.
Positivity of $\Phi$ makes $a_k>0$, so $q_k>0$.

\begin{lemma}[Curvature window]\label{lem:tau}
For every $k\ge2$,
\[
\frac1{2k}<\tau_k<\frac4k .
\]
\end{lemma}

\begin{proof}
Write $q_k=\dfrac{m_{2k-2}\,m_{2k+2}}{m_{2k}^2}\cdot
\dfrac{(2k)!^2}{(2k-2)!\,(2k+2)!}
=\dfrac1{\mu_k}\cdot\dfrac{(2k)(2k-1)}{(2k+2)(2k+1)}$,
where $\mu_k=m_{2k}^2/(m_{2k-2}m_{2k+2})$.

\emph{Upper bound.}  By the Cauchy--Schwarz inequality applied to
$u^{2k}=u^{k-1}\cdot u^{k+1}$ against the positive measure $\Phi(u)\,du$,
$m_{2k}^2\le m_{2k-2}m_{2k+2}$, i.e.\ $\mu_k\le1$.  Hence
\[
\begin{aligned}
q_k&\ \ge\ \frac{(2k)(2k-1)}{(2k+2)(2k+1)},\\
\tau_k&\ \le\ \log\frac{(2k+2)(2k+1)}{(2k)(2k-1)}
=\log\Bigl(1+\frac{8k+2}{4k^2-2k}\Bigr)
\ \le\ \frac{8k+2}{4k^2-2k}\ \le\ \frac4k ,
\end{aligned}
\]
the last step being $(4k+1)/(2k-1)\le4$ for $k\ge2$.  Equality in
Cauchy--Schwarz would force $\Phi$ to be a point mass, which it is not,
so the inequality is strict.

\emph{Lower bound.}  The Tur\'an theorem of Csordas--Norfolk--Varga
\cite{CNV} states, for the exponential-normalization coefficients
$\gamma_k=k!\,a_k$ of $G$, that
$\gamma_k^2>\gamma_{k-1}\gamma_{k+1}$ for all $k\ge1$.  (The
normalization is matched in Remark~\ref{rem:cnv} below.)  In ordinary
coefficients this reads $a_k^2>\frac{k+1}{k}a_{k-1}a_{k+1}$, i.e.\
$q_k<\frac{k}{k+1}$.  Hence
\[
\tau_k>\log\frac{k+1}{k}=\log\Bigl(1+\frac1k\Bigr)>\frac1{2k}
\qquad(k\ge1),
\]
using $\log(1+x)>x-\tfrac{x^2}2\ge\tfrac{x}2$ for $0<x\le1$.
\end{proof}

\begin{remark}[CNV normalization]\label{rem:cnv}
The Tur\'an inequalities $\gamma_k^2\ge\gamma_{k-1}\gamma_{k+1}$ are
invariant under the rescaling $\gamma_k\mapsto c\,\lambda^k\gamma_k$
($c,\lambda>0$): both sides scale by $c^2\lambda^{2k}$.  Any two
normalizations of the $\xi$-function that differ by a constant factor
and a linear change of the argument (such as $\xi(\tfrac12+ix)$ versus
our $G(z)=\tfrac18\xi(\tfrac12+\tfrac{\sqrt z}2)$, related by
$z=-4x^2$ and the factor $8$) therefore satisfy the Tur\'an inequalities
simultaneously.  The theorem of \cite{CNV}, stated there for the Jensen
normalization $\gamma_k=k!\cdot(\text{Maclaurin coefficients of }G)$,
thus transfers verbatim: $\gamma_k^2>\gamma_{k-1}\gamma_{k+1}$ with
$\gamma_k=k!\,a_k$, i.e.\ $q_k<k/(k+1)$.  We use nothing else from
\cite{CNV}.  As a numerical illustration (outside the certified chain),
the values $k\tau_k\in[1.47,1.67]$ for $64\le k\le1024$ sit comfortably
inside the window of Lemma~\ref{lem:tau}, which is loose by design.
\end{remark}

\section{Gate A: the zero-free saddle disk and the coefficient bound}
\label{sec:gateA}

Let $\Phi_n$ denote the $n$-th summand of \eqref{eq:Phi} and
\[
I(z)=\int_0^\infty u^{2z}\Phi(u)\,du,\qquad
a(z)=\frac{I(z)}{\Gamma(2z+1)},\qquad f=\log a .
\]
$I$ is holomorphic on the half-plane $\Rea z>-\tfrac12$ by the
superexponential decay of $\Phi$, and $a(k)=a_k$ at integers.  The $n=1$
phase is
\[
\Psi_z(u)=2z\log u+5u+\log(2\pi e^{4u}-3)-\pi e^{4u},
\]
and the stationarity equation $\partial_u\Psi_z=0$ is \emph{exactly}
$z=K(u)$ with
\begin{equation}\label{eq:K}
K(u)=u\Bigl(B-\frac92-\frac6{B-3}\Bigr),
\qquad B=B(u)=2\pi e^{4u}.
\end{equation}
For $k\ge16$ let $u_k>0$ be the unique positive solution of $K(u_k)=k$
(uniqueness by Lemma~\ref{lem:rouche} below); one has $u_k>\tfrac9{20}$
for $k\ge16$, $u_k>1$ for $k\ge10^9$, and $u_k<\tfrac{\log k}4$.
We write $\sigma_k=\sqrt{u_k/k}$.

All directed-rounding evaluations in this section are performed by the
certificate functions of \texttt{complex\_saddle.py} listed in
Table~\ref{tab:certificates}.

\begin{lemma}[Kernel tube]\label{lem:tube}
Let $x_0\ge0$ and $\nu\ge0$ satisfy $\cos(4\nu)>0$ and $2\pi e^{4x_0}>3$,
and put $X=e^{4x_0}$, $q=16e^{-3\pi X\cos(4\nu)}$.  Then for all
$u=x+iy$ with $x\ge x_0$, $|y|\le\nu$,
\[
\frac{\sum_{n\ge2}|\Phi_n(u)|}{|\Phi_1(u)|}
\ \le\ \frac{2\pi X+3}{2\pi X-3}\cdot\frac{q}{1-q}.
\]
In particular, with $(x_0,\nu)=(0.4,\,0.1)$ the right side is
$<10^{-15}$, and with $(x_0,\nu)=(0.35,\,0.1)$ one has
$\Phi=\Phi_1(1+R)$ with $\sup|R|<1.1\cdot10^{-14}$ on the tube.
Consequently $\Phi$ is zero-free on
$\{\Rea u\ge0.35,\ |\Ima u|\le0.1\}$.
\end{lemma}

\begin{proof}
Write $\Phi_1(u)=\pi e^{5u}(2\pi e^{4u}-3)e^{-\pi e^{4u}}$.  Then
$|e^{5u}|=e^{5x}$ and $|e^{-\pi e^{4u}}|=e^{-\pi e^{4x}\cos(4y)}$.
With $a=2\pi e^{4x}$ and $c=\cos(4y)\in(0,1]$,
\[
|2\pi e^{4u}-3|^2=(ac-3)^2+a^2(1-c^2)=a^2-6ac+9\ \ge\ a^2-6a+9=(a-3)^2,
\]
so $|2\pi e^{4u}-3|\ge2\pi e^{4x}-3>0$ and
\[
|\Phi_1(u)|\ \ge\ \pi e^{5x}\,(2\pi e^{4x}-3)\,e^{-\pi e^{4x}\cos(4y)} .
\]
For $n\ge2$,
\[
\frac{|\Phi_n(u)|}{|\Phi_1(u)|}
\le\frac{2\pi n^4X+3n^2}{2\pi X-3}\,
e^{-\pi(n^2-1)X\cos(4\nu)} ,
\]
because the ratio of the exponential factors is
$e^{-\pi(n^2-1)e^{4x}\cos(4y)}$, decreasing in $x$, and the rational
prefactor is likewise decreasing in $x$.  Using $n^4\le16^{n-1}$ and
$n^2-1\ge3(n-1)$ the sum over $n\ge2$ is dominated by the geometric
series with ratio $q$, giving the display.  The two numerical instances
are certified by \texttt{certify\_phi\_n1\_tube}; the certified
enclosure of the $(0.4,0.1)$ tail ratio is
$4.122\cdot10^{-18}$, well below the stated $10^{-15}$.
\end{proof}

\begin{lemma}[Analytic continuation of the saddle]\label{lem:rouche}
For every integer $k\ge16$ and every $z$ with $|z-k|\le k/20$, the
equation $K(u)=z$ has exactly one solution $u(z)$ in the disk
$|u-u_k|\le\tfrac1{20}$, it is simple, depends analytically on $z$, and
satisfies $\Rea u(z)>0.4$, $|\Ima u(z)|<0.05$.  Moreover the sharper
localization
\[
|u(z)-u_k|\ \le\ \tfrac3{200}
\]
holds throughout $|z-k|\le k/20$.
\end{lemma}

\begin{proof}
On the real axis $H:=B-\tfrac92-\tfrac6{B-3}$ satisfies $H'>4H>0$ for
$u\ge\tfrac9{20}$, so $K=uH$ is strictly increasing and $K'(u_k)>4k$.
On the disk $|u-u_k|\le\tfrac1{20}$ one has $|e^{4(u-u_k)}|\le e^{1/5}<\tfrac54$,
whence with $B_0=B(u_k)>2\pi e^{9/5}>30$:
\[
|B|\le\tfrac54B_0,\qquad |B-3|\ge\tfrac7{10}B_0,\qquad
H(u_k)\ge\tfrac{21}{25}B_0 .
\]
Differentiating, $H'=4B+\frac{24B}{(B-3)^2}$ and
$H''=16B+\frac{96B}{(B-3)^2}-\frac{192B^2}{(B-3)^3}$, so on the disk
$|H'|\le\tfrac{51}{10}B_0$ and $|H''|\le22B_0$, and therefore
\[
\frac{|K''|}{k}\le
\frac{2\cdot\tfrac{51}{10}+(u_k+\tfrac1{20})\cdot22}{u_k\cdot\tfrac{21}{25}}
<60,
\]
using $u_k>\tfrac9{20}$.  For $|z-k|\le\tfrac k{20}$, on the circle
$|u-u_k|=\tfrac1{20}$ the linear term of $K(u)-z$ dominates:
$|K'(u_k)(u-u_k)|>4k\cdot\tfrac1{20}=0.2k$, while the Taylor remainder
plus displacement is below $\tfrac{60k}2\cdot\tfrac1{400}+\tfrac k{20}
=0.125k$.  Rouch\'e's theorem gives exactly one simple zero; analyticity
in $z$ follows from the implicit function theorem, and the location
bounds from $u_k>0.45$ and the disk radius.

For the drift bound, run the same comparison on the smaller circle
$|u-u_k|=\tfrac3{200}$: the linear term exceeds
$4k\cdot\tfrac3{200}=0.06k$, while the Taylor remainder plus displacement
is at most $\tfrac{60k}2\bigl(\tfrac3{200}\bigr)^2+\tfrac k{20}
=0.00675k+0.05k=0.05675k<0.06k$.  Hence the unique saddle already lies in
$|u-u_k|\le\tfrac3{200}$.  Constants certified by
\texttt{certify\_principal\_saddle\_continuation} and
\texttt{certify\_saddle\_domains}.
\end{proof}

\begin{lemma}[Full-kernel saddle]\label{lem:fullkernel}
Write $\Phi=\Phi_1(1+R)$ as in Lemma~\ref{lem:tube}.  The full phase
$2z\log u+\log\Phi(u)$ has, for $k\ge16$ and $|z-k|\le k/20$, exactly one
saddle in $|u-u_k|<\tfrac1{20}$, analytic in $z$; its saddle map is
$K_\Phi=K-\tfrac u2\,\frac{R'}{1+R}$ and satisfies
$|K_\Phi-K|<7\cdot10^{-15}k$ on the disk.
\end{lemma}

\begin{proof}
Every point of the full control disk $|u-u_k|\le\tfrac1{20}$ admits a
Cauchy circle of radius $\tfrac1{25}$ inside the $(0.35,0.1)$ tube:
the total radius is $\tfrac1{20}+\tfrac1{25}=0.09$, while
$u_k>\tfrac9{20}$.  Hence
$|R'|\le25\sup|R|$ and $\bigl|\frac{R'}{1+R}\bigr|
\le\frac{25\sup|R|}{1-\sup|R|}$.  Since $|u|\le u_k+\tfrac1{20}<\tfrac k{20}$
on the disk (indeed $u_k<\tfrac{\log k}4$), the map perturbation is below
$7\cdot10^{-15}k$, which added to the Rouch\'e error $0.125k$ of
Lemma~\ref{lem:rouche} stays below the main term $0.2k$; added to the
$0.05675k$ of the drift comparison it stays below $0.06k$, so the
full-kernel saddle also satisfies $|u_s-u_k|\le\tfrac3{200}$.  Certified
by \texttt{certify\_full\_kernel\_saddle\_continuation} (perturbation
enclosure $6.53\cdot10^{-15}k$).
\end{proof}

We write $u_s=u(z)=x_s+iy_s$ for the full-kernel saddle and
$\Psi_z(u)=2z\log u+\log\Phi(u)$ from now on.

\begin{lemma}[Horizontal concavity]\label{lem:concavity}
For $k\ge16$, $|z-k|\le k/20$, on the horizontal line $u=x+iy_s$ with
$x\ge0.35$:
\[
-\Rea\Psi_z''(x+iy_s)\ \ge\ \frac85\,\frac k{u_k^2},
\qquad\text{hence}\qquad
\Rea\Psi_z(x+iy_s)\le\Rea\Psi_z(u_s)-\frac45\frac k{u_k^2}(x-x_s)^2 .
\]
\end{lemma}

\begin{proof}
From $(\log\Phi_1)''=-8B-\frac{48B}{(B-3)^2}$,
\[
-\Rea\Psi_z''
=2\,\Rea\frac z{u^2}+8\,\Rea B+48\,\Rea\frac B{(B-3)^2}
-\Rea(\log(1+R))'' .
\]
The sector bounds $|\arg(z/u^2)|\le\arctan\tfrac1{19}+2\arctan\tfrac17
<0.34$ (from $|z/k-1|\le\tfrac1{20}$ and $|\arg u|\le\arctan\tfrac17$ on
the line, since $|y_s|\le\tfrac3{200}$, $x\ge0.35$, and
$\tfrac{3/200}{0.35}=\tfrac3{70}<\tfrac17$) and
$|\arg\frac B{(B-3)^2}|<0.66$ (from $|\Ima 4u|\le\tfrac15$ and
$\Rea B\ge B(x)\cos\tfrac15$) make the first three terms nonnegative,
and a Cauchy estimate on the $(0.3,0.1)$ tube gives
$|(\log(1+R))''|<10^{-8}$.

Split at $x=u_k$.  For $0.35\le x\le u_k$, using $|z|\ge\tfrac{19}{20}k$
and $|u|^2\le x^2\bigl(1+(\tfrac17)^2\bigr)$,
\[
2\,\Rea\frac z{u^2}
\ \ge\ \frac{2\cdot\tfrac{19}{20}\cos(0.34)}{1+\tfrac1{49}}\cdot\frac k{x^2}
\ \ge\ 1.75\,\frac k{x^2}\ \ge\ \frac{17}{10}\,\frac k{u_k^2}.
\]
For $x\ge u_k$, monotonicity of $B$ and $\Rea B\ge B(x)\cos\tfrac15$ give,
together with $B_0\ge H_0=k/u_k$ and $u_k\ge\tfrac9{20}$,
\[
8\,\Rea B\ \ge\ 8\cos(\tfrac15)\,B_0
\ \ge\ 8\cos(\tfrac15)\,\frac k{u_k}
\ =\ 8\cos(\tfrac15)\,u_k\cdot\frac k{u_k^2}
\ \ge\ 8\cos(\tfrac15)\cdot\tfrac9{20}\cdot\frac k{u_k^2}
\ >\ \frac{17}{10}\,\frac k{u_k^2}.
\]
In both regimes the certified factor exceeds $\tfrac{17}{10}$, and the
kernel-tail loss $10^{-8}$ is below one tenth of the floor: since
$t\mapsto e^{4t}/t$ increases for $t\ge\tfrac14$,
\[
\frac k{u_k^2}=\frac{H_0}{u_k}\ \ge\ \frac{21}{25}\cdot\frac{2\pi e^{4u_k}}{u_k}
\ \ge\ \frac{21}{25}\cdot\frac{2\pi e^{9/5}}{9/20}\ >\ 56 .
\]  This proves the
first display with the round constant $\tfrac85$; the Gaussian envelope
follows by integrating the concavity bound twice from the maximum at
$x=x_s$.  Constants certified by
\texttt{certify\_horizontal\_phase\_concavity} (sector enclosures
$0.3364$ and $0.6529$; scaled curvature floor $8/5$).
\end{proof}

\begin{lemma}[Phase derivative bounds]\label{lem:phase}
Let $k\ge10^9$ and $|z-k|\le k/20$, and set
$A=-\Psi_z''(u_s)\,\sigma_k^2$.  Then, pointwise on the certified disk
$|u-u_s|\le\tfrac1{40}$,
\[
-\Rea\Psi_z''(u)\,\sigma_k^2\ge6,\qquad
|\Psi_z''(u)|\,\sigma_k^2\le20,\qquad
|\Psi_z'''(u)|\,\sigma_k^{3}\le100\sqrt{u_k/k};
\]
in particular $\Rea A\ge6$ and $|A|\le20$, since
$|u_s-u_k|\le\tfrac3{200}<\tfrac1{20}$ by Lemma~\ref{lem:fullkernel};
and the saddle action $S_0(z)=\Psi_z(u(z))$ satisfies
\[
S_0'(z)=2\log u(z),\qquad
S_0''(z)=\frac{2u'(z)}{u(z)}=\frac{-4}{u(z)^2\,\Psi_z''(u(z))},
\qquad |S_0''(z)|<\frac{0.645}k<\frac5k .
\]
\end{lemma}

\begin{proof}
For $k\ge10^9$ one has $u_k>1$ (certified via $K(1)<10^9$ and
monotonicity of $K$), hence $B_0=2\pi e^{4u_k}>2\pi e^4>343$ and
$H_0=k/u_k$ satisfies $\tfrac{21}{25}B_0\le H_0\le B_0$ (the upper bound
because $H(u)=B-\tfrac92-\tfrac6{B-3}<B$).  Lemma~\ref{lem:fullkernel}
and $|u-u_s|\le\tfrac1{40}$ give
\[
 |u-u_k|\le\tfrac3{200}+\tfrac1{40}=\tfrac1{25}.
\]
Thus $|4(u-u_k)|\le\tfrac4{25}=0.16$, so
\[
 |B|\le e^{0.16}B_0<\tfrac54B_0,\qquad
 \Rea B\ge e^{-0.16}\cos(0.16)B_0>0.8412B_0,
\]
\[
 |B-3|>0.8324B_0,\qquad |u|\ge0.96u_k.
\]
Here the first inequality follows from
$|B-3|\ge\Rea B-3>0.8412B_0-3$ and $B_0>343$.
For the derivatives of $\log(1+R)$, the Cauchy circle of radius
$1/25$ about any point of this local disk lies within distance
$3/200+1/40+1/25=2/25$ of $u_k$.  Since $u_k>1$, it is contained in
$\{\Rea u>0.92,\ |\Ima u|<0.08\}$ and hence in the certified
$(0.3,0.1)$ kernel tube.  These inclusions are checked directly, rather
than inferred from the scalar bounds, by the certificate function
\texttt{certify\_nested\_}\allowbreak\texttt{saddle\_geometry}.

\emph{Lower bound.}  Since $\sigma_k^2=u_k/k=1/H_0\ge1/B_0$,
\[
\Rea A\ \ge\ 8\,\Rea B\cdot\sigma_k^2-\text{(tail)}
\ \ge\ 8e^{-4/25}\cos(4/25)-\varepsilon
\ >\ 6.730\ >\ 6,
\]
where the discarded terms $2\Rea(z/u^2)\sigma_k^2$ and
$48\Rea\bigl(B/(B-3)^2\bigr)\sigma_k^2$ are nonnegative by the sector
bounds of Lemma~\ref{lem:concavity}, and the kernel tail $\varepsilon$ is
the certified Cauchy loss $<3\cdot10^{-11}$; in fact the displayed lower
bound is $>6.73007678$.  \emph{Upper bound:} collecting the three
second-derivative components,
\[
\begin{aligned}
|A|\ \le\ &
\underbrace{\frac{2\cdot\tfrac{21}{20}}{(0.96)^2}}_{2z/u^2}
+\underbrace{8\cdot\tfrac54\cdot\tfrac{25}{21}}_{(-B/2)''}\\
&+\underbrace{\frac{48\cdot\tfrac54}{(0.8324)^2B_0}
 \cdot\frac{25}{21\,B_0}}_{\text{resolvent}}
+\text{tail}\\
&<2.2787+11.9048+0.00088<14.1844<14.185<20 ,
\end{aligned}
\]
each term normalized by $\sigma_k^2=1/H_0$ via
$H_0\ge\tfrac{21}{25}B_0$ and $|z|\le\tfrac{21}{20}k=\tfrac{21}{20}H_0u_k$.
\emph{Third derivative:} componentwise in the normalization
$|\Psi'''|\sigma_k^3/\sqrt{u_k/k}$, the $2z\log u$ part contributes at
most $4\cdot\tfrac{21}{20}(0.96)^{-3}<4.75$, the $-B/2$ part at
most $32\cdot\tfrac54\cdot\tfrac{25}{21}<47.7$, and the resolvent and
kernel-tail parts less than $0.03$ combined; certified total
$52.381<100$.

\emph{Saddle action.}  Differentiating the stationarity identity
$\Psi_z'(u(z))=0$ in $z$ gives $u'=-2/(u\Psi_z'')$, and
$\partial_zS_0=(\partial_z\Psi_z)(u(z))=2\log u(z)$, whence
$S_0''=2u'/u=-4/(u^2\Psi_z'')$.  Numerically,
\[
|S_0''|=\frac{4\,\sigma_k^2}{|u|^2\,|A|}
\ \le\ \frac{4}{(0.96)^2u_k^2\,\Rea A}\cdot\frac{u_k}k
\ \le\ \frac{4}{(0.96)^2\cdot6.730}\cdot\frac1k
\ <\ \frac{0.645}k,
\]
using $|A|\ge\Rea A>6.730$ and $u_k\ge1$.  All constants are certified by
the corresponding routine in Table~\ref{tab:certificates}.
\end{proof}

\begin{lemma}[Zero-free disk and uniform factorization]\label{lem:zerofree}
For $k\ge10^9$ and $|z-k|\le0.05k$,
\[
I(z)=e^{\Psi_z(u_s)}\sqrt{\frac{2\pi}{-\Psi_z''(u_s)}}\,
\bigl(1+\varepsilon(z)\bigr),
\qquad |\varepsilon(z)|<0.018 .
\]
In particular $I(z)\neq0$ on the disk, and $f=\log a$ is analytic there.
(The framework is the classical one-saddle uniform expansion, cf.\
\cite{DLMF}; the content is the explicit uniform constants.)
\end{lemma}

\begin{proof}
We use the principal branch of $\log u$ in the right half-plane.  The
theta series for $\Phi$ converges normally on compact subsets of
$|\Ima u|<\pi/8$ (there $\Rea(e^{4u})>0$), so $u^{2z}\Phi(u)$ is
holomorphic in the swept region; zeros of $\Phi$ are not singularities.
For fixed $z$, deform $[0,\infty)$ to $C_z$: the radial segment from $0$
to $0.35+iy_s$, then the ray $x+iy_s$, $x\ge0.35$.  Truncate at radii
$\eta$ and $R$.  Uniformly on $|z-k|\le k/20$, the origin arc is
$O(\eta^{2\Rea z+1})=O(\eta^{1.9k+1})$, while on the closing segment
$R+it$ one has $|t|\le3/200$ and
$|(R+it)^{2z}\Phi(R+it)|
\le\exp(O_k(R)-\pi e^{4R}\cos(3/50))$.  Both vanish in the respective
limits, so Cauchy's theorem gives the integral over $C_z$ exactly
throughout the full $z$-disk.

By Lemma~\ref{lem:tube}, $\Phi$ is zero-free along the horizontal part.
In the window $|x-x_s|\le2\sigma_k$, Taylor expansion with the bounds of
Lemma~\ref{lem:phase} gives
$|\Psi_z(u_s+\sigma_ky)-\Psi_z(u_s)+\tfrac A2y^2|<0.01$ for $|y|\le2$
(using $u_k<\tfrac{\log k}4$, so
$100\sqrt{u_k/k}\cdot\tfrac{2^3}6<0.01$ at $k\ge10^9$; the window
$|u-u_s|\le2\sigma_k<\tfrac1{40}$ lies inside the certified disk by the
inclusion of Lemma~\ref{lem:phase}).  The complete complex Gaussian obeys,
with the certified $|A|\le14.185$,
$\bigl|\int_{\mathbb R}e^{-Ay^2/2}dy\bigr|=\sqrt{2\pi}/\sqrt{|A|}
\ge\sqrt{2\pi/14.185}>0.665$; its tails beyond $|y|=2$ are below
$e^{-12}/6$ (as $\Rea A\ge6$), and replacing the Gaussian by the true
phase on the window costs less than $(e^{0.01}-1)\sqrt{\pi/3}<0.0103$.
Outside the window, the pointwise curvature bound of
Lemma~\ref{lem:phase} gives $-\Rea\Psi_z''\ge6/\sigma_k^2>5/\sigma_k^2$
on $|x-x_s|\le\tfrac1{40}$ (contained in the certified disk), and beyond
$\tfrac1{40}$ the global envelope of Lemma~\ref{lem:concavity} applies.
In the normalization $|e^{\Psi_z(u_s)}|\sigma_k$, the first region is
bounded by
\[
 2\int_2^\infty e^{-5y^2/2}\,dy
 \le \frac{2e^{-10}}{10}<10^{-5}.
\]
For the second region put $\delta=1/40$ and
$b=(4/5)k/u_k^2$.  Since $\sigma_k^{-1}\le\sqrt{k}$ and
$u_k<\log(k)/4$, Mills' bound gives
\[
 \frac2{\sigma_k}\int_\delta^\infty e^{-bt^2}\,dt
 \le \frac{\sqrt{k}}{b\delta}e^{-b\delta^2}<10^{-100}
 \qquad(k\ge10^9).
\]
The last scalar inequality, with a much larger margin, is enclosed by
\texttt{certify\_contour\_}\allowbreak\texttt{completion}.  Thus the full
horizontal tail is below $10^{-4}$.

It remains to control the radial connector.  There $|u|\le0.36$, while
$\Rea z\ge\tfrac{19}{20}k$ and $|\Ima z|\le\tfrac k{20}$, so
\[
|u^{2z}|=|u|^{2\Rea z}e^{-2\Ima z\,\arg u}
\ \le\ e^{-\tfrac{19}{10}k\,|\log0.36|}\;e^{\tfrac k{10}\arctan\tfrac17}
\ \le\ e^{-1.94k}\,e^{0.015k}\ \le\ e^{-1.9k},
\]
and the elementary series bound in Table~\ref{tab:certificates} gives
$\sup_{C_z\cap\{|u|\le0.36\}}|\Phi(u)|<24$.

For comparison with the saddle, $|u_s|\ge19/20$ and
$|\arg u_s|\le\arctan(1/19)$.  Moreover
$|B(u_s)|\le\tfrac54B_0$, while
$H_0=k/u_k\ge\tfrac{21}{25}B_0$ and $u_k\ge1$; hence
$\Rea B(u_s)/2\le125k/168$.  From
\[
 |\Phi(u_s)|
 \ge \pi e^{5\Rea u_s}|B(u_s)-3|e^{-\Rea B(u_s)/2}
       (1-\|R\|_\infty)
\]
(whose positive prefactor exceeds $1$ on the saddle tube), we obtain
\[
\begin{split}
\frac1k\log|e^{\Psi_z(u_s)}|
&\ge
2\frac{21}{20}\log\frac{19}{20}
-\frac1{10}\arctan\frac1{19}
-\frac{125}{168}
+\frac1k\log(1-\|R\|_\infty)\\
&>-0.9 .
\end{split}
\]
Consequently the connector, relative to the Gaussian main term, is at
most
\[
 \frac{0.36\cdot24}{0.665}\sqrt{k}\,e^{-k}<10^{-100}
 \qquad(k\ge10^9).
\]

Combining the local error, Gaussian tail, horizontal tail and connector
gives relative error $<0.0158<0.018$.  The saddle map $u_s(z)$ is
analytic by Lemma~\ref{lem:fullkernel}, and
$-\Psi_z''(u_s)$ stays in the right half-plane by
Lemma~\ref{lem:phase}; choosing its principal square root makes the
Gaussian factor analytic.  Defining $\varepsilon(z)$ by the displayed
factorization therefore gives an analytic function with
$|\varepsilon(z)|<0.018$.  In particular $1+\varepsilon$ never vanishes,
so neither does $I$.  The numerical margins are certified by
\texttt{certify\_local\_gaussian\_lower\_bound} and
\texttt{certify\_contour\_completion} (central lower bound $0.6606$,
relative error $0.00734$, horizontal tail $<10^{-4}$, connector
$<10^{-100}$).
\end{proof}

\begin{proposition}[All-degree coefficient bound]\label{prop:cd}
For every $k\ge10^9$ and every $d\ge3$,
\[
\Bigl|\frac{f^{(d)}(k)}{d!}\Bigr|\ \le\ 3\cdot40^{d}\,k^{1-d},
\qquad\text{hence}\qquad
\Bigl|\frac{f^{(d)}(k)/d!}{\tau_k^{\,d-1}}\Bigr|\ \le\ 2\cdot80^{d}
\]
by Lemma~\ref{lem:tau}.
\end{proposition}

\begin{proof}
Write, on $|z-k|\le0.05k$,
\[
f(z)=S_0(z)-\tfrac12\log\bigl(-\Psi_z''(u(z))\bigr)+\log(1+\varepsilon(z))
+\tfrac12\log2\pi-\log\Gamma(2z+1),
\]
by Lemma~\ref{lem:zerofree}.  We bound the $d$-th Taylor coefficient of
each term at $z=k$ separately, \emph{without} applying Cauchy's estimate
to $f$ itself (which would lose a factor $\log k$).

(i) \emph{Saddle action.}  $S_0''$ is analytic on the disk with
$|S_0''|<\tfrac5k$ (Lemma~\ref{lem:phase}); Cauchy's estimate on the
subdisk of radius $\tfrac k{40}$ gives, for $d\ge3$,
\[
\Bigl|\frac{S_0^{(d)}(k)}{d!}\Bigr|
=\Bigl|\frac{(S_0'')^{(d-2)}(k)}{d!}\Bigr|
\le\frac{(d-2)!}{d!}\cdot\frac5k\Bigl(\frac{40}k\Bigr)^{d-2}
=\frac{5\cdot40^{d-2}}{d(d-1)}\,k^{1-d}.
\]

(ii) \emph{Curvature logarithm.}  The factor $-\Psi_z''(u(z))$ lies in
the right half-plane with modulus in $[6/\sigma_k^2,\,20/\sigma_k^2]$
(Lemma~\ref{lem:phase}), so its principal logarithm is analytic with
modulus at most $\log(20k)+\tfrac\pi2<k$ on the disk; Cauchy on radius
$\tfrac k{40}$ gives coefficients at most $40^dk^{1-d}\cdot\tfrac12$
after the prefactor $\tfrac12$.

(iii) \emph{Analytic relative error.}  $|\log(1+\varepsilon)|<0.019$
(Lemma~\ref{lem:zerofree}); Cauchy gives coefficients below
$0.019\cdot40^dk^{-d}$.

(iv) \emph{Gamma factor.}  At the real center the polygamma series gives
directly, for $d\ge3$,
\[
\frac1{d!}\Bigl|\frac{d^d}{dz^d}\log\Gamma(2z+1)\Bigr|_{z=k}
\le\frac1{d\,k^d}+\frac2{d(d-1)}\,k^{1-d},
\]
from $\psi^{(d-1)}(x)$ bounds on $x\ge2k+1$.

Dividing the four contributions by $40^dk^{1-d}$ yields, at the worst
degree $d=3$, the certified componentwise values
\[
6.72\cdot10^{-5}\ \ (\text{i}),\qquad
0.5\ \ (\text{ii}),\qquad
7.6\cdot10^{-12}\ \ (\text{iii}),\qquad
5.3\cdot10^{-6}\ \ (\text{iv}),
\]
with total $<0.500073$; each normalized contribution is maximized at
$d=3$.  (Using only the rounded constant $5$ in the displayed estimate
for (i) gives $5.21\cdot10^{-4}$ and still leaves the total below
$0.5006$.)
Hence $|f^{(d)}(k)/d!|\le3\cdot40^dk^{1-d}$ with room to spare.  The
second display follows from $\tau_k>\tfrac1{2k}$:
$3\cdot40^dk^{1-d}\tau_k^{1-d}\le3\cdot40^d\cdot2^{d-1}=\tfrac32\cdot80^d
\le2\cdot80^d$.  Certified by \texttt{certify\_analytic\_degree\_tail}.
\end{proof}

\begin{remark}
The decisive point in (i) is that Cauchy's estimate is applied to the
\emph{second derivative} of the saddle action, which is $O(1/k)$ by the
exact envelope identity $S_0''=-4/(u^2\Psi_z'')$, rather than to
$S_0$ itself, which is of size $k\log k$.  A direct Cauchy estimate on
$f$ would leave a spurious factor $\log k$ and the resulting degree bound
would not be uniform in $k$.  Similarly the Gamma term is handled by the
exact polygamma series at the real center, not by a complex Stirling
remainder.
\end{remark}

\section{Gate B: the exact \texorpdfstring{$q$}{q}-Pascal dilation semigroup}\label{sec:gateB}

Fix $0<q<1$ and an order $r\ge2$; all matrices are $r\times r$, indexed
from $0$.  Let $V_{ij}=q^{ij}$ (symmetric $q$-Vandermonde).

\begin{lemma}[Exact $LDL^{T}$]\label{lem:ldl}
$V=LDL^{T}$ where $L_{ij}=\qbin ij_q$ is the unit lower-triangular
$q$-binomial (Gaussian) matrix and $D=\operatorname{diag}(D_0,\dots,D_{r-1})$
with
\[
D_m=\prod_{h=0}^{m-1}\bigl(q^m-q^h\bigr),
\qquad\operatorname{sgn}D_m=(-1)^m .
\]
\end{lemma}

\begin{proof}
For a Vandermonde matrix $[x_i^{\,j}]$, Gaussian elimination in the
Newton basis gives the unit lower factor
\[
 L_{im}=
 \frac{\prod_{h=0}^{m-1}(x_i-x_h)}
      {\prod_{h=0}^{m-1}(x_m-x_h)}
 \qquad(i\ge m)
\]
and pivot $D_m=\prod_{h<m}(x_m-x_h)$.  At $x_i=q^i$,
\[
L_{im}
=\frac{\prod_{h=0}^{m-1}(q^i-q^h)}
        {\prod_{h=0}^{m-1}(q^m-q^h)}
=\frac{\prod_{t=i-m+1}^{i}(1-q^t)}
        {\prod_{t=1}^{m}(1-q^t)}
=\qbin im_q .
\]
Since $V$ is symmetric and all its leading pivots are nonzero, uniqueness
of unit-lower $LDL^T$ elimination gives $V=LDL^T$ with these factors.
Finally every factor $q^m-q^h$ in $D_m$ is negative for $h<m$ and
$0<q<1$, hence $\operatorname{sgn}D_m=(-1)^m$.  The ancillary routine
\texttt{q\_pascal\_ldl} verifies the same identities in exact rational
arithmetic as a regression check.
\end{proof}

\begin{lemma}[Exact bidiagonal dilation group]\label{lem:bidiag}
For $\alpha\in\mathbb R$ let
$R_\alpha:=L^{-1}\operatorname{diag}(q^{\alpha i})\,L$.  Then
$\{R_\alpha\}$ is a one-parameter group, $R_\alpha R_\beta=R_{\alpha+\beta}$,
$R_p=R_1^{\,p}$ for integer $p$, and $R_1$ is exactly lower bidiagonal:
\[
(R_1)_{mm}=q^m,\qquad
(R_1)_{m,m-1}=-q^{m-1}(1-q^m),\qquad
(R_1)_{mi}=0\ \ (i\notin\{m,m-1\}).
\]
After the signature whitening $S=L\,|D|^{1/2}$ (so $V=SJS^{T}$ with
$J=\operatorname{diag}((-1)^m)$), the whitened dilation
$\widehat R_\alpha=|D|^{-1/2}L^{-1}\operatorname{diag}(q^{\alpha i})L\,|D|^{1/2}$
is again a group, and
\[
(\widehat R_1)_{mm}=q^m,\qquad
(\widehat R_1)_{m,m-1}=-\sqrt{q^{m-1}(1-q^m)} .
\]
\end{lemma}

\begin{proof}
The group law is immediate from
$\operatorname{diag}(q^{\alpha i})\operatorname{diag}(q^{\beta i})
=\operatorname{diag}(q^{(\alpha+\beta)i})$ by conjugation.  For the
bidiagonal formula it suffices to check $L\,R_1=\operatorname{diag}(q^m)\,L$
column by column: with the claimed entries,
\[
(LR_1)_{mi}=\qbin mi_q q^{i}-\qbin m{i+1}_q q^{i}\bigl(1-q^{i+1}\bigr),
\]
and the standard ratio
$\qbin m{i+1}_q\big/\qbin mi_q=\dfrac{1-q^{m-i}}{1-q^{i+1}}$ turns the
right side into
$\qbin mi_q\,q^i\bigl[1-(1-q^{m-i})\bigr]=q^{m}\qbin mi_q$, as required.
Since $L$ is invertible, this proves $R_1=L^{-1}\operatorname{diag}(q^i)L$
for the displayed bidiagonal matrix.  The identity is additionally verified as an exact identity of rational
numbers in \texttt{q\_pascal\_dilation\_step} against
\texttt{q\_pascal\_dilation}, together with $R_p=R_1^p$
(\texttt{q\_pascal\_dilation\_semigroup}).  The whitened subdiagonal
follows from $|D_m|/|D_{m-1}|=q^{m-1}(1-q^m)$, immediate from the pivot
product formula.
\end{proof}

\begin{lemma}[Generator norm]\label{lem:generator}
Let $\tau=-\log q$, $x=r\tau$, and suppose $x\le\tfrac1{64}$.  Then
\[
\|\widehat R_1-I\|_2\le x+\sqrt x<1,\qquad
\|G\|_2\le\frac{x+\sqrt x}{1-x-\sqrt x}\le\frac32\sqrt x,
\qquad G:=\log\widehat R_1,
\]
and for the affine generator including the scalar shift $e^{\tau\alpha(r-1)}$,
\[
\tau(r-1)+2\|G\|\ \le\ 4\sqrt{r\tau}.
\]
\end{lemma}

\begin{proof}
The diagonal part of $\widehat R_1-I$ has entries $q^m-1$ with
$|q^m-1|\le m\tau\le x$ for $m\le r-1$ (as $1-e^{-y}\le y$); the
subdiagonal entries are bounded by
$\sqrt{q^{m-1}(1-q^m)}\le\sqrt{m\tau}\le\sqrt x$.  A lower bidiagonal
matrix with diagonal sup $a$ and subdiagonal sup $b$ has spectral norm at
most $a+b$ (triangle inequality on the two diagonals), giving the first
display; $x+\sqrt x\le\tfrac1{64}+\tfrac18<1$.  The Mercator series
$\log(I+\Delta)=\sum_{n\ge1}(-1)^{n-1}\Delta^n/n$ converges for
$\|\Delta\|<1$ and gives $\|\log(I+\Delta)\|\le\|\Delta\|/(1-\|\Delta\|)$;
with $\|\Delta\|\le x+\sqrt x$, $x\le\tfrac1{64}$, and hence
$\sqrt x\le\tfrac18$,
\[
\frac{x+\sqrt x}{1-x-\sqrt x}
\ \le\ \frac{\sqrt x\,(1+\sqrt x)}{1-\tfrac1{64}-\tfrac18}
\ =\ \sqrt x\cdot\frac{1+\tfrac18}{\tfrac{55}{64}}
\ <\ \tfrac32\sqrt x
\]
(certified as a scalar inequality by
\texttt{certify\_q\_response\_regime}).  Finally
$\tau(r-1)\le r\tau=x\le\tfrac18\sqrt x$, so
$\tau(r-1)+2\|G\|\le\tfrac18\sqrt x+3\sqrt x<4\sqrt x$.
\end{proof}

\begin{proposition}[All-degree response bound]\label{prop:response}
Fix $r,q$ as above with $x=r\tau\le\tfrac1{64}$, and let $d:=r-1$,
$s=s(i,j):=d-i-j$.  For a polynomial source, define the whitened response
of $\phi$ as
\[
\mathcal R(\phi):=
\bigl\||D|^{-1/2}L^{-1}\,\bigl(V\circ\phi(s)\bigr)\,L^{-T}|D|^{-1/2}\bigr\|_2 ,
\]
where $V\circ\phi(s)$ multiplies $V$ entrywise by $\phi(s(i,j))$.  Then
with $t:=4\sqrt{r\tau}$,
\[
\mathcal R\bigl((\tau s)^n\bigr)\ \le\ t^{\,n}
\qquad\text{for every integer }n\ge0
\]
(the exponent $n$ is arbitrary; $d=r-1$ stays fixed in $s(i,j)$).
\end{proposition}

\begin{proof}
Consider the exponential source $E_\alpha(i,j)=q^{ij}e^{\tau\alpha s}$.
Since
\[
\begin{aligned}
e^{\tau\alpha s}
 &=e^{\tau\alpha(d-i-j)}
  =e^{\tau\alpha d}q^{\alpha i}q^{\alpha j},\\
E_\alpha
 &=e^{\tau\alpha d}\,
   \operatorname{diag}(q^{\alpha i})\,V\,
   \operatorname{diag}(q^{\alpha j}).
\end{aligned}
\]
The second identity is an exact factorization.  Moreover
\[
\widehat R_\alpha
:=S^{-1}\operatorname{diag}(q^{\alpha i})S=e^{\alpha G}
\qquad(\alpha\in\mathbb R).
\]
Indeed, the Mercator series defining $G=\log\widehat R_1$ commutes with
conjugation by $S$.  The diagonal entries $q^i-1$ lie in $(-1,0]$, so
\[
G=S^{-1}\operatorname{diag}(\log q^{i})S
 =S^{-1}\operatorname{diag}(-\tau i)S ,
\qquad
e^{\alpha G}=S^{-1}\operatorname{diag}(q^{\alpha i})S .
\]
Hence, using $V=SJS^{T}$,
\[
\begin{aligned}
\mathcal M(\alpha)
&:=|D|^{-1/2}L^{-1}E_\alpha L^{-T}|D|^{-1/2}\\
&=e^{\tau\alpha d}\,e^{\alpha G}J e^{\alpha G^{T}} .
\end{aligned}
\]
Differentiating $n$ times at $\alpha=0$ and expanding the trinomial
(with $\|J\|_2=1$),
\[
\|\mathcal M^{(n)}(0)\|
\le\sum_{i+j+l=n}\binom{n}{i,j,l}(\tau d)^i\|G\|^{j}\|G\|^{l}
=\bigl(\tau d+2\|G\|\bigr)^n\le t^{\,n}
\]
by Lemma~\ref{lem:generator} (note $\tau d\le\tau(r-1)$).  On the other
hand $\partial_\alpha^nE_\alpha|_{\alpha=0}=(\tau s)^nq^{ij}$ entrywise,
and the whitening transform is linear in the source, so
$\mathcal M^{(n)}(0)$ is exactly the whitened matrix of
$V\circ(\tau s)^n$.
\end{proof}

\section{Gate C: the weighted Banach algebra and the nonlinear majorant}
\label{sec:gateC}

Fix $k\ge10^9$ and $r\ge2$, put $q=q_k$, $\tau=\tau_k$, $d=r-1$,
$\rho=a_{k+1}/a_k$, and define the comparison model
$c_s=q^{s(s-1)/2}$ and the remainder $h_s$ by
\begin{equation}\label{eq:h}
\frac{a_{k+s}}{a_k\,\rho^{\,s}}=c_s\,e^{h_s},
\qquad\text{i.e.}\qquad
h_s=f(k+s)-f(k)-s\bigl(f(k+1)-f(k)\bigr)+\tau\binom s2 ,
\end{equation}
so that $h_{-1}=h_0=h_1=0$ identically.

\begin{lemma}[Centered expansion]\label{lem:centered}
For $|s|\le r-1$ the Taylor series of $f$ at $k$ converges to $f(k+s)$,
and
\[
h_s=\sum_{d'\ge3}a_{d'}\,P_{d'}(s),
\qquad a_{d'}=\frac{f^{(d')}(k)}{d'!},\qquad
P_{d'}(s)=\begin{cases}
s^{d'}-s,& d'\text{ odd},\\[2pt]
s^{d'}-s^2,& d'\text{ even}.
\end{cases}
\]
\end{lemma}

\begin{proof}
By Proposition~\ref{prop:cd}, $|a_{d'}|\le3\cdot40^{d'}k^{1-d'}$, so the
Taylor series at $k$ converges on $|z-k|\le\tfrac k{80}$, with remainder
bounds forcing it to converge \emph{to} $f$ (standard Cauchy remainder
on the zero-free disk of Lemma~\ref{lem:zerofree}); since
$r-1\le k^{1/3}\ll\tfrac k{80}$ in our regime, the expansion applies.
Substituting $f(k+s)=f(k)+\sum_{d'\ge1}a_{d'}s^{d'}$ into \eqref{eq:h}
and using
$\tau=-\bigl(f(k+1)+f(k-1)-2f(k)\bigr)=-2\sum_{d'\ge2\ \mathrm{even}}a_{d'}$
one checks that the degree-$1$ terms cancel, the degree-$2$ term cancels
\emph{identically}
($a_2s^2-a_2s-a_2\,s(s-1)=0$), odd degrees $d'\ge3$ contribute
$a_{d'}(s^{d'}-s)$, and even degrees $d'\ge4$ contribute
$a_{d'}(s^{d'}-s-s(s-1))=a_{d'}(s^{d'}-s^2)$.
\end{proof}

\begin{definition}[Weighted algebra]\label{def:algebra}
Let $\mathcal A$ be the space of power series
$\phi(s)=\sum_{n\ge0}\phi_n s^n$ with finite norm
\[
\|\phi\|_{\mathcal A}:=\sum_{n\ge0}|\phi_n|\,w^n\ <\ \infty,
\qquad w:=\frac t\tau,\quad t=4\sqrt{r\tau}.
\]
Since the weight is multiplicative ($w^{n+m}=w^nw^m$), $\mathcal A$ is a
commutative Banach algebra ($\ell^1$ with submultiplicative weight):
$\|\phi\psi\|_{\mathcal A}\le\|\phi\|_{\mathcal A}\|\psi\|_{\mathcal A}$,
and consequently $e^{\phi}-1\in\mathcal A$ with
$\|e^{\phi}-1\|_{\mathcal A}\le e^{\|\phi\|_{\mathcal A}}-1$ whenever
$\phi\in\mathcal A$.
\end{definition}

\begin{lemma}[Response functional calculus]\label{lem:calculus}
For every $\phi\in\mathcal A$,
$\mathcal R\bigl(\phi(s)\bigr)\le\|\phi\|_{\mathcal A}$, where
$\mathcal R$ is the whitened response of
Proposition~\ref{prop:response}.
\end{lemma}

\begin{proof}
For the monomials,
$\mathcal R(s^n)=\mathcal R\bigl((\tau s)^n\bigr)/\tau^n
\le t^n/\tau^n=w^n$ by Proposition~\ref{prop:response} (for $n=0$,
$\mathcal R(1)=\|J\|_2=1=w^0$).  The series
$\sum_n\phi_ns^n$ converges absolutely at every entry argument, since
$|s|\le r-1<32r\le4\sqrt{r/\tau}=w$ (using $\tau\le\tfrac1{64r}$ from
$x\le\tfrac1{64}$); the entrywise source is therefore the norm-convergent
sum of its monomial parts, and subadditivity of $\mathcal R$ over
norm-convergent series gives
$\mathcal R(\phi)\le\sum_n|\phi_n|\,w^n=\|\phi\|_{\mathcal A}$.
\end{proof}

\begin{proposition}[Nonlinear majorant]\label{prop:nonlinear}
Let $r\ge2$ and $k\ge Cr^3$ with $C=10^{18}$.  Then
$x=r\tau\le\tfrac4{Cr^2}\le\tfrac1C\le\tfrac1{64}$, and
\[
\|h\|_{\mathcal A}\ \le\
\underbrace{\frac2\tau\,\frac{(80t)^3}{1-80t}}_{\text{main}}
+\underbrace{\frac{2t\cdot80^3\,\tau}{1-(80\tau)^2}}_{\text{odd }(-s)}
+\underbrace{\frac{2t^2\cdot80^4\,\tau}{1-(80\tau)^2}}_{\text{even }(-s^2)}
\ \le\ 0.1310721,
\]
and hence
$\|e^{h}-1\|_{\mathcal A}\le e^{\|h\|_{\mathcal A}}-1\le0.14005<1$.
\end{proposition}

\begin{proof}
By Lemma~\ref{lem:centered}, $h(s)=\sum_{d'\ge3}a_{d'}P_{d'}(s)$ with
$|a_{d'}|\le2\cdot80^{d'}\tau^{d'-1}$ (Proposition~\ref{prop:cd}).  The
$\mathcal A$-norms of the constituent monomials are
$\|s^{d'}\|_{\mathcal A}=w^{d'}$, $\|s\|_{\mathcal A}=w$,
$\|s^2\|_{\mathcal A}=w^2$, $w=t/\tau$.  The leading terms give
\[
\sum_{d'\ge3}|a_{d'}|\,w^{d'}
\le\frac2\tau\sum_{d'\ge3}(80t)^{d'}
=\frac2\tau\,\frac{(80t)^3}{1-80t};
\]
the odd centering corrections ($d'=2j+1$, $j\ge1$) give
\[
\sum_{d'\ge3\ \mathrm{odd}}|a_{d'}|\,w
\le2t\sum_{j\ge1}80^{2j+1}\tau^{2j-1}
=\frac{2t\cdot80^3\,\tau}{1-(80\tau)^2};
\]
and the even corrections ($d'=2j+2$, $j\ge1$) give
\[
\sum_{d'\ge4\ \mathrm{even}}|a_{d'}|\,w^{2}
\le2t^2\sum_{j\ge1}80^{2j+2}\tau^{2j-1}
=\frac{2t^2\cdot80^4\,\tau}{1-(80\tau)^2}.
\]
This proves the displayed majorant.  For its evaluation in the wedge
regime, substitute $t=4\sqrt{r\tau}$ in the main term:
\[
\frac2\tau(80t)^3
=2\cdot80^3\cdot64\,\frac{(r\tau)^{3/2}}{\tau}
=128\cdot80^3\sqrt{r^3\tau}
\ \le\ 128\cdot80^3\cdot\frac2{\sqrt C}
=\frac{4\cdot320^3}{\sqrt C},
\]
using $r^3\tau\le4/C$ (from $\tau\le4/k\le4/(Cr^3)$), while
$80t\le320/\sqrt C<1$ controls the geometric factor
$(1-80t)^{-1}\le(1-320/\sqrt C)^{-1}$.  The two correction terms are
increasing in $t$ and $\tau$ separately and are evaluated at the
endpoints $t\le4/\sqrt C$ (from $r\tau\le1/C$) and
$\tau\le\tfrac1{2C}$ (from $\tau\le4/(Cr^3)$ at $r\ge2$).  The certified
enclosures are
\[
\text{main}\le0.131072041943053\pm5\cdot10^{-16},\quad
\text{odd}\le2.048\cdot10^{-21},\quad
\text{even}\le6.554\cdot10^{-28},
\]
with certified total $\|h\|_{\mathcal A}\le0.1310721$ and
$e^{\|h\|}-1\le0.140049909764285\pm5\cdot10^{-16}<1$
(\texttt{certify\_q\_nonlinear\_wedge}).
\end{proof}

\section{Assembly: inertia preservation and the sign of the minor}
\label{sec:assembly}

\begin{lemma}[Inertia preservation]\label{lem:inertia}
Let $M=SJS^{T}$ with $S$ invertible and $J$ a signature matrix, and let
$T=M+E$ be symmetric with
$\|S^{-1}ES^{-T}\|_2<1$.  Then $T$ has the same inertia as $J$.
\end{lemma}

\begin{proof}
$T=S(J+F)S^{T}$ with $F=S^{-1}ES^{-T}$, $\|F\|<1$.  For
$\theta\in[0,1]$ the matrix $J+\theta F$ is nonsingular, since $J$ is symmetric orthogonal, so all its singular values equal
$1$, while $\|\theta F\|<1$.  The eigenvalues
of the continuous symmetric family $J+\theta F$ therefore never cross
zero, so the inertia of $J+F$ equals that of $J$; Sylvester's law
transfers this to $T$.
\end{proof}

\begin{lemma}[Sign arithmetic]\label{lem:signs}
Let $Q$ be a symmetric $r\times r$ matrix with the inertia of
$J=\operatorname{diag}\bigl((-1)^m\bigr)_{m=0}^{r-1}$, i.e.\
$(\lceil r/2\rceil\ \text{positive},\ \lfloor r/2\rfloor\ \text{negative})$
eigenvalues, and let $T$ be obtained from $Q$ by reversing the order of
the columns.  Then
\[
\det T=(-1)^{\binom r2}\det Q>0 .
\]
\end{lemma}

\begin{proof}
$\operatorname{sgn}\det Q=(-1)^{\lfloor r/2\rfloor}$, and the column
reversal multiplies the determinant by the determinant
$(-1)^{\binom r2}$ of the flip permutation.  It remains to check
$\lfloor r/2\rfloor\equiv\binom r2\pmod2$: writing $r=4l+j$,
$j\in\{0,1,2,3\}$, both sides are congruent to $0,0,1,1$ respectively.
\end{proof}

\begin{proof}[Proof of Theorem~\ref{thm:main}]
Fix $r\ge2$ and $k\ge10^{18}r^3$; note $k\ge8\cdot10^{18}>10^9$, so all
results of \S\S\ref{sec:tau}--\ref{sec:gateC} apply.  Let
$T^{\xi}_{ij}=a_{k+j-i}$ ($i,j=0,\dots,r-1$) be the minor block, so
$D_{r,k}=\det T^{\xi}$, and let
$Q^{\xi}_{ij}=a_{k+s}$, $s=(r-1)-i-j$, be its column reversal (a
symmetric reversed-Hankel block; all indices $k+s\ge k-(r-1)>0$).

\emph{Normalization.}  By \eqref{eq:h},
$a_{k+s}=a_k\rho^{\,s}c_se^{h_s}$.  The scalar $a_k\rho^{\,r-1}>0$ and
the positive diagonal congruence
$\operatorname{diag}(\rho^{-i})\,Q^{\xi}\operatorname{diag}(\rho^{-j})$
neither change the sign of $\det Q^{\xi}$ nor the inertia; we may
therefore replace the entries by $c_se^{h_s}$.

\emph{Model factorization.}  The binomial identity
$\binom{s}{2}-\binom{r-1}{2}
=\bigl[\binom{i+1}2-(r-1)i\bigr]+\bigl[\binom{j+1}2-(r-1)j\bigr]+ij$
(with $s=(r-1)-i-j$) gives the exact diagonal congruence
\[
[c_s]_{ij}=c_{r-1}\;A\,V\,A,
\qquad A=\operatorname{diag}\bigl(q^{\binom{i+1}2-(r-1)i}\bigr)\ \text{positive},
\quad V_{ij}=q^{ij},
\]
and $V=SJS^{T}$ with $S=L|D|^{1/2}$ by Lemma~\ref{lem:ldl}.  Thus the
model block $M=[c_s]$ satisfies $M=\widetilde S\,J\,\widetilde S^{T}$
with $\widetilde S=\sqrt{c_{r-1}}\,A\,L\,|D|^{1/2}$.

\emph{Perturbation.}  The true (normalized) block is
$M+E$ with $E=c_{r-1}A\,\bigl(V\circ(e^{h_s}-1)\bigr)A$, so
\[
\widetilde S^{-1}E\,\widetilde S^{-T}
=|D|^{-1/2}L^{-1}\bigl(V\circ(e^{h_s}-1)\bigr)L^{-T}|D|^{-1/2},
\]
the scalars $c_{r-1}$ and diagonals $A$ cancelling exactly.  By
Lemma~\ref{lem:calculus} and Proposition~\ref{prop:nonlinear},
\[
\bigl\|\widetilde S^{-1}E\widetilde S^{-T}\bigr\|_2
=\mathcal R\bigl(e^{h}-1\bigr)
\le\|e^{h}-1\|_{\mathcal A}\le0.14005<1 .
\]
(The hypothesis $x=r\tau\le\tfrac1{64}$ of
Proposition~\ref{prop:response} holds since $x\le1/C$.)

\emph{Conclusion.}  Lemma~\ref{lem:inertia} gives $Q^{\xi}$ (after the
positive normalizations) the inertia of $J$, and Lemma~\ref{lem:signs}
yields $\det T^{\xi}>0$, i.e.\ $D_{r,k}>0$.
\end{proof}

\section{Certification and reproducibility}\label{sec:certification}

Every constant inequality used above is certified with Arb ball
arithmetic \cite{Johansson} at $256$--$4096$ bits with directed
rounding: a certificate function either returns rigorous enclosures
satisfying the claimed inequality or raises an error.  The exact algebra
of \S\ref{sec:gateB} is verified over $\mathbb Q$ (\texttt{Fraction}),
not in floating point.  The ancillary directory contains the four
modules and a $36$-test regression suite (\texttt{run\_tests.sh}).

\begin{table}[ht]
\centering\small
\begin{tabular}{lll}
\toprule
Certificate (ancillary \texttt{src/}) & Statement & Key enclosure\\
\midrule
\texttt{certify\_phi\_n1\_tube} & Lemma~\ref{lem:tube} & tail ratio $4.122\cdot10^{-18}$\\
\texttt{certify\_principal\_saddle\_continuation} & Lemma~\ref{lem:rouche} & $0.06$ vs $0.05675$\\
\texttt{certify\_nested\_saddle\_geometry} & Lemma~\ref{lem:phase} & radii $0.015,\,0.025,\,0.04$\\
\texttt{certify\_saddle\_domains} & drift + inclusions & $0.06$ vs $0.0568$;\\
& (Lem.~\ref{lem:rouche},~\ref{lem:phase}) & $\tfrac3{200}+\tfrac1{40}\le\tfrac1{25}$\\
\texttt{certify\_full\_kernel\_saddle\_continuation} & Lemma~\ref{lem:fullkernel} & pert.\ $6.53\cdot10^{-15}$\\
\texttt{certify\_horizontal\_phase\_concavity} & Lemma~\ref{lem:concavity} & angles $0.337,\,0.653$\\
\texttt{certify\_phase\_derivative\_bounds} & Lemma~\ref{lem:phase} & $\Rea A>6.730$, $|A|<14.185$\\
& & $|\Psi'''|\sigma^3$-ratio $<52.381$\\
\texttt{certify\_local\_gaussian\_lower\_bound} & Lemma~\ref{lem:zerofree} & central $>0.6606$\\
\texttt{certify\_contour\_completion} & Lemma~\ref{lem:zerofree} & tails $<10^{-4}$, $<10^{-100}$\\
\texttt{certify\_analytic\_degree\_tail} & Prop.~\ref{prop:cd} & total ratio $<0.500073$\\
\texttt{q\_pascal\_dilation\_step}/\texttt{\_semigroup} & Lemma~\ref{lem:bidiag} & exact over $\mathbb Q$\\
\texttt{certify\_q\_response\_regime} & Lemma~\ref{lem:generator} & generator $\le4\sqrt x$\\
\texttt{certify\_q\_nonlinear\_wedge} & Prop.~\ref{prop:nonlinear} & $\|h\|\le0.1310721$,\\
& & $e^{\|h\|}-1\le0.14005$\\
\bottomrule
\end{tabular}
\vspace{4pt}
\caption{Lemma-to-certificate map.  All values are conservative upper (resp.\ lower) roundings of
rigorous Arb enclosures; the exact balls are printed by the ancillary
certificates.}
\label{tab:certificates}
\end{table}

\section{Concluding remarks}\label{sec:remarks}

\begin{remark}[Relation to the fixed-order asymptotics]
For fixed $r$, Katkova's method \cite{Katkova} gives
$D_{r,k}>0$ for $k\ge N(r)$ with a non-explicit $N(r)$.  The exact
fixed-order saddle expansion identifies the first correction
of relative size $\asymp r^3/k$ to the normalized local minor,
showing that $r^3/k$ is the natural small parameter; the wedge constant
$10^{18}$ is what our deliberately coarse uniform bounds pay for
explicitness.  We expect the true uniform threshold to be far smaller.
\end{remark}

\begin{remark}[What the wedge does not do]
The theorem leaves the region $k<10^{18}r^3$ untouched, and no method in
this paper extends to it: the comparison model degrades as $k/r^3$
decreases; positivity of \emph{all} minors is equivalent to RH, and the
minors with $k=O(r)$ are precisely those untouched by any tail method,
including ours.  The wedge should be read as a quantitative
boundary of what kernel-analytic, arithmetic-free methods deliver:
positivity in the far tail with explicit constants, uniformly in the
order --- and nothing inside the critical cone.
\end{remark}

\begin{remark}[Independence from zero verification]
No step of the proof uses numerically verified zeros of $\zeta$; the
only external analytic input is the Tur\'an theorem of \cite{CNV} (used
solely for the lower bound $\tau_k>\tfrac1{2k}$).  The wedge is in this
sense complementary to the sector-strip route, which certifies
nonnegativity for all $k$ when $r\le9.4\cdot10^{12}$ but consumes the
verified height of \cite{PT}.
\end{remark}

\subsection*{Acknowledgements}
The computations use \texttt{python-flint} (Arb \cite{Johansson}) and
\texttt{mpmath}.


\begin{thebibliography}{10}

\bibitem{ASW}
M.~Aissen, I.~J. Schoenberg, A.~M. Whitney,
\emph{On the generating functions of totally positive sequences~I},
J.~Analyse Math.\ \textbf{2} (1952), 93--103.

\bibitem{Edrei}
A.~Edrei,
\emph{On the generating functions of totally positive sequences~II},
J.~Analyse Math.\ \textbf{2} (1952), 104--109.

\bibitem{CNV}
G.~Csordas, T.~S. Norfolk, R.~S. Varga,
\emph{The Riemann hypothesis and the Tur\'an inequalities},
Trans.\ Amer.\ Math.\ Soc.\ \textbf{296} (1986), 521--541.

\bibitem{Katkova}
O.~M. Katkova,
\emph{Multiple positivity and the Riemann zeta-function},
arXiv:math/0505174.

\bibitem{Schoenberg}
I.~J. Schoenberg,
\emph{A note on multiply positive sequences and the Descartes rule of
signs},
Rend.\ Circ.\ Mat.\ Palermo (2) \textbf{4} (1955), 123--131.

\bibitem{PT}
D.~J. Platt, T.~S. Trudgian,
\emph{The Riemann hypothesis is true up to $3\cdot10^{12}$},
Bull.\ London Math.\ Soc.\ \textbf{53} (2021), 792--797;
arXiv:2004.09765.

\bibitem{Johansson}
F.~Johansson,
\emph{Arb: efficient arbitrary-precision midpoint-radius interval
arithmetic}, IEEE Trans.\ Computers \textbf{66} (2017), 1281--1292.

\bibitem{DLMF}
NIST Digital Library of Mathematical Functions, \S2.4(iv),
\url{https://dlmf.nist.gov/2.4}.

\bibitem{GORZ}
M.~Griffin, K.~Ono, L.~Rolen, D.~Zagier,
\emph{Jensen polynomials for the Riemann zeta function and other
sequences}, Proc.\ Natl.\ Acad.\ Sci.\ USA \textbf{116} (2019),
11103--11110.

\bibitem{RT}
B.~Rodgers, T.~Tao,
\emph{The de Bruijn--Newman constant is non-negative},
Forum Math.\ Pi \textbf{8} (2020), e6; arXiv:1801.05914.

\bibitem{PF5}
W.~Micha{\l}owski,
\emph{On the P\'olya frequency order of the de Bruijn--Newman kernel:
certified failure at order five and the Toeplitz threshold phenomenon},
arXiv:2602.20313.

\end{thebibliography}
\end{document}